\documentclass[a4paper]{article}
\usepackage[utf8]{inputenc}
\usepackage[english]{babel}
\usepackage{paralist,url,verbatim,anysize}
\usepackage{amscd}
\usepackage{amsmath,amsthm,amssymb,amsfonts}
\marginsize{2.9 cm}{2 cm}{1.9 cm}{2 cm}
\usepackage[usenames,dvipsnames,svgnames,table]{xcolor}
\definecolor{purple}{HTML}{961C8C}
\usepackage{setspace}
\usepackage{mathrsfs}
\usepackage{xparse}
\usepackage{suffix}
\usepackage{microtype}
\usepackage{manfnt}
\usepackage{nicefrac}
\usepackage{mathabx}
\usepackage{graphicx,graphics}
\usepackage[bookmarksopen=true, backref=page]{hyperref}
\hypersetup{colorlinks=true, citecolor=blue}
\usepackage{epstopdf}
\usepackage[mathlines, pagewise]{lineno}
\usepackage{bbm}
\usepackage[font=small,labelfont=bf]{caption}
\usepackage[labelfont=rm, textfont=rm, labelformat=simple]{subcaption}

\theoremstyle{plain}
\newtheorem{theorem}{\bf Theorem}[subsection]

\newtheorem{cor}[theorem]{Corollary}
\newtheorem{lemma}[theorem]{Lemma}
\newtheorem{prp}[theorem]{Proposition}
\newtheorem{problem}[theorem]{Problem}

\newtheorem{thmnonumber}{\bf Theorem}

\theoremstyle{definition}
\newtheorem{rem}[theorem]{Remark}
\newtheorem{definition}[theorem]{Definition}

\newtheorem{example}[theorem]{Example}

\DeclareMathAlphabet{\mathpzc}{OT1}{pzc}{m}{it}
\newenvironment{frcseries}{\fontfamily{frc}\selectfont}{}

\newcommand{\conv}{\operatorname{conv} }
\renewcommand{\S}{{\hspace{.12em} \mathrm{S}}}

\newcommand{\R}{\mathbb{R}}

\newcommand{\cm}[1]{}
\newcommand\Defn[1]{\emph{\color{RubineRed}#1}}
\newcommand{\tu}{{\, {\vartriangle}\, }}
\newcommand{\td}{{\, {\triangledown}\, }}
\newcommand{\sd}{\mathrm{sd}}
\newcommand{\sm}{\hspace{.08em}}
\newcommand{\Lk}{\mathrm{Lk}\sm}
\newcommand{\St}{\mathrm{St}\sm}
\newcommand{\intx}{\mathrm{int}\sm}
\newcommand{\rint}{\mathrm{relint}\sm} 
\newcommand{\parti}{\partial \sm}
\newcommand{\cl}{\mathrm{cl}\sm}

\newcommand{\cst}{\operatorname{c-st}}
\newcommand{\csd}{\operatorname{c-sd}}
\newcommand{\RS}{\mathrm{R}}
\newcommand\F{\mathrm{F}}
\newcommand{\st}{\mathrm{st}\sm}

\newcommand{\TT}{\mathrm{T}}
\newcommand{\RN}{\mathrm{N}}

\begin{document}

\author{Karim A. Adiprasito
\thanks{This work was supported by the DFG within the research training group ``Methods for Discrete Structures'' (GRK1408) and by the Romanian NASR, CNCS – UEFISCDI, project PN-II-ID-PCE-2011-3-0533.}\\ 
\small Institut des Hautes \'Etudes Scientifiques\\ 
\small Le Bois-Marie 35, route de Chartres\\ 
\small 91440 Bures-sur-Yvette, France\\
\small \url{adiprasito@math.fu-berlin.de}
 \and
 \and
 Bruno Benedetti \thanks{Supported by the Swedish Research Council, grant ``Triangulerade M{\aa}ngfalder, Knutteori i diskrete Morseteori'', and by the DFG grant ``Discretization in Geometry and Dynamics''.} \\
 \small Institut f\" ur Informatik\\ \small Freie Universit\"at Berlin\\
 \small Takustrasse 9, 14195 Berlin\\
\small \url{bruno@zedat.fu-berlin.de}
}

\date{July 10, 2013}
\title{Subdivisions, shellability, and collapsibility of products}
\maketitle
\enlargethispage{3mm}
\begin{abstract}
We prove that the second derived subdivision of any rectilinear triangulation of any convex polytope is shellable. Also, we prove that the first derived subdivision of every rectilinear triangulation of any convex $3$-dimensional polytope is shellable. This complements Mary Ellen Rudin's classical example of a non-shellable rectilinear triangulation of the tetrahedron. Our main tool is a new relative notion of shellability that characterizes the behavior of shellable complexes under gluing. 

As a corollary, we obtain a new characterization of the PL property in terms of shellability: A triangulation of a sphere or of a ball is PL if and only if it becomes shellable after sufficiently many derived subdivisions. This improves on results by Whitehead, Zeeman and Glaser, and answers a question by Billera and Swartz. 

We also show that any contractible complex can be made collapsible by repeatedly taking products with an interval. This strengthens results by Dierker and Lickorish, and resolves a conjecture of Oliver. Finally, we give an example that this behavior extends to non-evasiveness, thereby answering a question of Welker. 
\end{abstract}

\vskip 3mm

\paragraph*{\large Shellability}\hspace{-3mm} is one of the earliest notions in combinatorial topology, cf.\ \cite{Grun, RourkeSanders, Z}. It provides a combinatorial analogue of the topological notion of the PL ball: Every shellable contractible manifold is a PL ball. However, the converse is false; some PL balls are not shellable \cite{LME}.
One of the  main goals of this paper is to show that a simple combinatorial manipulation, applied repeatedly, makes every PL ball shellable.

Our starting point is the famous \emph{non-shellable} subdivision of the tetrahedron provided by Rudin~\cite{Rudin}, which gained much influence since its publication \cite{Bing, PB, LSM, Schenzel, Stanley, VV}. This example somewhat limited the appeal of the shellability property because it demonstrates that even simple triangulations of a trivial geometry may fail to satisfy it.
In our first main result, we will argue that Rudin's ball should not discourage us too much: while subdivisions of polytopes are not shellable in general, they are very close to being shellable.

\begin{thmnonumber}[Corollary \ref{cor:s} and Theorem \ref{thm:sh3}]\label{mainthm:ShellableSub}
If $C$ is any subdivision of a convex polytope, the second derived subdivision of $C$ is shellable. If $\dim C=3$, already the first derived subdivision of $C$ is shellable.
\end{thmnonumber}

The bound is best possible at least in dimension $3$, as Rudin's ball is not shellable. For subdivisions of $d$-polytopes, $d\ge 4$, one needs either one or two derived subdivisions to guarantee shellability, and we conjecture the former to hold true. As for the case $d\le 2$, it is not hard to see that no derived subdivision is needed: every triangulation of the $d$-disk, $d\le 2$, is shellable.

Derived subdivisions preserve plenty of combinatorial properties, including collapsibility and shellability. Theorem~\ref{mainthm:ShellableSub}, which extends results in \cite{AD2013}, shows that they even induce these properties after a relatively small number of steps. In contrast, PL balls can be arbitrarily nasty when $d \ge 3$: For each $m$ and every $d\ge 3$ there is a PL $d$-ball whose $m$-th derived subdivision is still not shellable~\cite{GOO}. As a consequence of Theorem~\ref{mainthm:ShellableSub}, we obtain the following result:  

\begin{thmnonumber}[{Theorem~\ref{thm:shellyyo}}] \label{mainthm:equivalence}
A simplicial complex is PL homeomorphic to a shellable complex if and only if it becomes shellable after finitely many derived subdivisions. 
\end{thmnonumber}

In particular, a triangulated ball or sphere is PL if and only if it becomes shellable after a finite number of derived subdivisions, cf.\ Corollary~\ref{cor:characterizationPL2}. This resolves a problem by Ed Swartz and Lou Billera (personal communication). A~similar result holds if we replace ``shellable'' with ``collapsible'':

\begin{thmnonumber}[{Theorem~\ref{thm:zeemanyes}}] \label{mainthm:equivalencec}
A simplicial complex is PL homeomorphic to a collapsible complex if and only if it becomes collapsible after finitely many derived subdivisions. 
\end{thmnonumber}

The two results simplify and generalize some foundational results in PL topology: Whitehead \cite[Thm.~7]{Whitehead} and Glaser \cite[pp.\ 58--69]{Glaser} proved that every PL ball admits a collapsible (iterated) stellar subdivision, and it is a fundamental result of Whitehead \cite[Thm.\ 3]{Whitehead} that this property characterizes PL balls among all PL manifolds. Theorem~\ref{mainthm:equivalencec} shows that stellar subdivisions can be replaced by the subclass of derived subdivisions.
Similarly, one can use Theorem~\ref{mainthm:equivalence} to characterize the PL notion for manifolds: Theorem~\ref{mainthm:equivalence} shows that a homology manifold $M$ is PL if and only if all vertex links of some iterated subdivision of $M$ are shellable, cf.\ Corollary~\ref{cor:characterizationPL2}.

Finally, we turn to a problem motivated by the Zeeman conjecture, which claims that the product of any contractible $2$-dimensional complex $C$ with the interval $[0,1]$ is PL homeomorphic to a collapsible complex. This begs the question: 
Is it true that every contractible polytopal $2$-complex, after multiplication with the interval $\mathrm{I}:=[0,1]$, becomes collapsible itself? We provide the following answer.

\begin{thmnonumber}[Cor.~\ref{cor:oliver}]\label{mainthm:oliver}
For any contractible complex $C$ there is an $n\ge 0$ such that $C \times \mathrm{I}^n$~is~collapsible.
\end{thmnonumber}\enlargethispage{3mm}

Contrary to the case of derived subdivisions, the analogous theorem for shellability does not hold (cf.\ Remark \ref{rem:shelloliver}). Theorem \ref{mainthm:oliver} was already conjectured by Bob Oliver in a 1998 written correspondence to Anders Bj\"orner. It furthermore strengthens results by Dierker \cite{Dierker} and Lickorish \cite[Lem.\ 1]{Lickorish} who proved that for every contractible complex, there is an $n\ge 0$ such that $C\times \mathrm{I}^n$ is PL homeomorphic to a collapsible complex. Using the intuition of Theorem \ref{mainthm:oliver}, we also give a concrete example of a evasive poset that becomes non-evasive when multiplied with the two-element total, thereby solving a question by Welker \cite[Open Problem~2]{Welker}, cf.\ Proposition~\ref{prop:orderComplex}.

To this day the Zeeman conjecture remains open, although some related problems have been solved \cite{Cohen, EDG} and the Poincar\'e--Hamilton--Perelman Theorem implies the Zeeman conjecture for a wide class of complexes \cite{Gillman}.  Via Theorem~\ref{mainthm:equivalencec}, we can rephrase it in a rather concrete way: The Zeeman conjecture is equivalent to the claim that for any contractible $2$-complex $C$ there is an $n\ge 0$ such that $\sd^n (C \times \mathrm{I})$ is collapsible (Cor.~\ref{prop:CohenBary}). The answer to our original question, however, is negative: For~any $m, n \ge 0$, we can describe a contractible $2$-complex $C_{m,n}$ such that the $m$-th derived subdivision of $C_{m,n} \times \mathrm{I}^n$ is \emph{not} collapsible (Cor.~\ref{cor:zeemanno})

\paragraph*{\large Notation.}
For polytopes and polytopal complexes, we follow the standard notation in the literature \cite{Grun, RourkeSanders, Z}, see also \cite{AD2013}. We use $\conv X$ to denote the \Defn{convex hull} of a set $X$, and $\cl X$, $\intx X$, $\rint X$ and $\parti X$ shall denote the \Defn{closure}, \Defn{interior}, \Defn{relative interior} and \Defn{boundary} of $X$ respectively.

By $\R^d$ and $S^d$ we denote the euclidean $d$-space and the unit sphere in $\R^{d+1}$, respectively. A \Defn{(euclidean) polytope} in $\R^d$ is the convex hull of finitely many points in $\R^d$. Similarly, a \Defn{spherical polytope} in $S^d$ is the convex hull of a finite number of points that all belong to some open hemisphere of $S^d$. Spherical polytopes are in natural one-to-one correspondence with euclidean polytopes via radial projections. A \Defn{polytopal complex} in $\R^d$ (resp.\ in $S^d$) is a finite collection of polytopes in $\R^d$ (resp.~$S^d$) that is closed under passing to faces of its elements and such that the intersection of any two polytopes is a face of both. 

A polytopal complex is \Defn{pure} if all its facets are of the same dimension. Two polytopal complexes $C,\, D$ are \Defn{combinatorially equivalent}, denoted by $C\cong D$, if their face posets are isomorphic. A polytopal complex combinatorially equivalent to $C$ is also called a \Defn{realization} of $C$. A polytopal complex is \Defn{simplicial} if all its faces are simplices. The set of $k$-dimensional faces of a polytopal complex $C$ is denoted by $\F_k(C)$, and the cardinality of this set is denoted by $f_k(C).$ If $C$ is a pure polytopal $d$-complex, then a \Defn{principal} subcomplex is any pure $d$-dimensional subcomplex or the empty complex.

The \Defn{underlying space} $|C|$ of a polytopal complex $C$ is the union of its faces. We will frequently abuse notation and identify a polytopal complex with its underlying space, as common in the literature. For instance, we do not distinguish between a polytope and the complex formed by its faces.

If $C$ is a polytopal complex, and $A$ is some set, we define the \Defn{restriction} $\RS(C,A)$ of $C$ to $A$ as the inclusion-maximal subcomplex $D$ of $C$ such that $D$ lies in $A$. The \Defn{star} of $\sigma$ in $C$, denoted by $\St(\sigma, C)$, is the minimal subcomplex of $C$ that contains all faces of $C$ containing $\sigma$. The \Defn{deletion} $C-D$ of a subcomplex $D$ from $C$ is the subcomplex of $C$ given by $\RS(C,  C{\sm\setminus\sm} \rint{D})$. 

Next, we define links with a differential-geometric approach, cf.\ \cite{AD2013}, \cite[Sec.\ 2.2]{DM-NP}. Let $p$ be any point of a metric space $X$. By $\TT_p X$ we denote the tangent space of $X$ at $p$. Let $\TT^1_p X$ be the restriction of $\TT_p X$ to unit vectors.  If $Y$ is any subspace of $X$, then $\RN_{(p,Y)} X$ denotes the subspace of the tangent space spanned by vectors orthogonal to $\TT_p Y \subset \TT_p X$, and we define $\RN^1_{(p,Y)} X:= \RN_{(p,Y)} X \cap \TT^1_p Y$.

If $\tau$ is any face of a polytopal complex $C$ in $X^d=\R^d$ (or $X^d=S^d$) containing a nonempty face $\sigma$ of $C$, then $\RN^1_{(p,\sigma)} \tau$ forms a spherical polytope isometrically embedded in $\RN^1_{(p,\sigma)} |C|\subset \RN^1_{(p,\sigma)} X^d$. Here, the space $\RN^1_{(p,\sigma)} X^d$ is isometric to a sphere of dimension $n-\dim \sigma -1$, and will be considered as such. The collection of all polytopes in $\RN^1_{(p,\sigma)} |C|$ obtained this way forms a polytopal complex denoted by $\Lk_p(\sigma, C)$, the \Defn{link} of $C$ at $\sigma$. Up to ambient isometry $\Lk_p(\sigma, C)$ and  $\RN^1_{(p,\sigma)} \tau$ in $ \RN^1_{(p,\sigma)} |C|\subset \RN^1_{(p,\sigma)} X^d$ do not depend on $p$, thus, the base-point $p$ will be omitted in notation whenever possible. By convention, we set $\Lk(\varnothing, C):=C$. If $C$ is simplicial, and $v$ is a vertex of $C$, then $\Lk(v,C)\cong(C-v)\cap \St(v,C)=\St(v,C)-v.$

Let $C$, $D$ be polytopal complexes, and let $C'$ and $D'$ be realizations of $C$ and $D$ in skew affine subspaces of $\R^d$. A \Defn{join} $C\ast D$ of $C$ and $D$ is the complex consisting of faces $\conv (\sigma \cup \tau),\ \sigma \in C',\ \tau \in D'.$
This is well-defined up to combinatorial equivalence. Caution: the join may not be realizable using $C$ and $D$ themselves: If $C$ and $D$ are complexes in $\R^2$ containing at least $3$ vertices each, then $C\ast D$ contains a $K_{3,3}$ and is therefore not realizable in $\R^2$. If $C$ is a simplicial complex, and $\sigma$, $\tau$ are faces of $C$, then $\sigma\ast \tau$ is combinatorially equivalent to the minimal face of $C$ containing both $\sigma$ and $\tau$ (assuming it exists), hence they will be identified. If $\sigma$ is a face of a simplicial complex $C$, and $\tau$ is a face of $\Lk(\sigma,C)$, then $\sigma \ast \tau$ is combinatorially equivalent to the face $\chi$ of $C$ with $\Lk(\sigma,\chi)=\tau$. Again, we will simply denote $\chi$ by $\sigma \ast \tau$.

Analogously, if $C$ and $D$ are polytopal complexes, and $C'$ and $D'$ are realizations of $C$ and $D$ in orthogonal subspaces of $\R^d$, then a \Defn{product} $C\times D$ of $C$ and $D$ is given as the complex consisting of faces $ \{(x,y)\in \R^d: x\in \tau,\, y\in \sigma\},\ \sigma \in C',\ \tau \in D'.$ Similar to the join, the product of two polytopal complexes is defined up to combinatorial equivalence only.

An \Defn{elementary collapse} is the deletion of a \Defn{free} face $\sigma$ from a polytopal complex~$C$, i.e.\ the deletion of a nonempty face $\sigma$ of $C$ that is strictly contained in only one other face of $C$. In this situation, we say that $C$ \Defn{(elementarily) collapses} onto $C-\sigma$, and write $C\searrow_{\mathrm{e}} C-\sigma.$ More generally, we say that the complex $C$ \Defn{collapses} to a subcomplex $C'$, and write~$C\searrow C'$, if $C$ can be reduced to $C'$ by a sequence of elementary collapses. A \Defn{collapsible} complex is a complex that collapses onto a single vertex.

Let $C$ be a pure polytopal complex with $N$ facets. If $\dim C=0$, a \Defn{shelling} of $C$ is any ordering of its vertices. If $\dim C = d >0$, a \Defn{shelling} for $C$ is an ordering $(F_1, \cdots, F_n)$, the \Defn{shelling order}, of its facets such that for each $i \in \{1, \ldots,n-1\}$ the complex $B:=F_i \cap \bigcup_{j={i+1}}^{n} F_j$ is pure $(d-1)$-dimensional and both $B$ and $\partial F_i-B$ are shellable or empty\footnote{We will see in Corollary~\ref{cor:eq} that this is equivalent to the definition of shellability classically used, cf.\ \cite[Sec.\ 8]{Z}.}.  A complex $C$ is \Defn{shellable} if it has a shelling. 

A \Defn{subdivision} of a polytopal complex $C$ is a polytopal complex $C'$ with the same underlying space as~$C$, such that for every face $F'$ of $C'$ there is some face $F$ of~$C$ for which $F' \subset F$. A \Defn{subdivision} of a set $M$ is a polytopal complex whose underlying space is $M$. Now, let $C$ denote any polytopal complex, and let $\tau$ denote any face of $C$. Let $v_\tau$ denote a point anywhere in the relative interior of $\tau$. Define
\[
\st(\tau,C):=(C-\tau) \cup \{\conv (\{v_\tau\}\cup \sigma) : \sigma \in \St(\tau,C)-\tau \}.
\]
The complex $\st(\tau,C)$ is a \Defn{stellar subdivision} of $C$ at $\tau$. A \Defn{derived subdivision} $\sd\sm C$ of a polytopal complex $C$ is any subdivision of $C$ obtained by stellarly subdividing at all faces in order of decreasing dimension of the faces of $C$, cf.\ \cite{Hudson}. 
  
Finally, recall that two polytopal complexes $C$ and $D$ are called \Defn{PL equivalent} (or \Defn{PL homeomorphic})  if some subdivision $C'$ of $C$ is combinatorially equivalent to some subdivision $D'$ of $D$. A \Defn{PL $d$-ball} is any polytopal complex that is PL homeomorphic to the $d$-simplex. A \Defn{PL $d$-sphere} is a polytopal complex that is PL homeomorphic to the boundary of the $(d+1)$-simplex. A polytopal complex $C$ is a \Defn{PL $d$-manifold} if for each vertex $v$ of~$C$ the complex $\Lk(v,C)$ is a PL $(d-1)$-sphere or a PL $(d-1)$-ball. 

\section{Derived subdivisions and shellability}
Subdivisions of $3$-polytopes are always collapsible \cite{CHIL}, but not always shellable \cite{Rudin}. In the second part of this section, we reconcile this by showing that the second derived subdivision of any convex $d$-complex is shellable. In the third part of this section, we draw some consequences for PL topology from Theorem~\ref{mainthm:ShellableSub}, among them Theorem~\ref{mainthm:equivalence}. We start by providing a theory of relative shellings.

\subsection{Shellability at the boundary}

The proof of Theorem~\ref{mainthm:ShellableSub}, in its key ideas, follows Chapter $2$ of the first author's thesis \cite{AD2013}, see also \cite{KarimBrunoMG&C}. However, shellability is a rather restrictive property, so that the proof is quite involved. We will therefore introduce some concepts first, among them the notion of relative shellings.

\paragraph*{Shelling sequences and relative shellings.} Let $C$ be any complex, and let $\tau$ be any face of $C$. We denote by $C\td \tau$ the complex $\RS(C, \cl(C{\sm\setminus\sm}\St(\tau,C)))$, the \Defn{removal} of $\tau$ from $C$. If $D\subset C$ is a subcomplex, we define \[C\td D=\bigcap_{\tau\ \text{facet of}\ D} C\td \tau.\]

\begin{definition}[Relative shellings, shelling sequences and shelling orders] \label{def:seq} Let $C$ be a polytopal $d$-complex, and let $D$ be a $(d-1)$-dimensional complex such that $D\cup C$ is a polytopal complex. Assume that $d (f_d(C)-1) > 0$, and let $P$ be any facet of $C$. The removal $C\rightarrow C':=C \td P$ is a \Defn{shelling step relative to~$D$} if $P\cap (C'\cup D)$ is pure of dimension $d-1$ and the complexes $P\cap (C'\cup D)$ and $\parti P \td (P\cap (C'\cup D))$ are shellable or empty. Given two pure polytopal $d$-complexes $C$ and $C'$ then \Defn{$C$  shells to $C'$ relative to $D$}, and write $C \searrow^D_{\S} C'$, if there is a sequence of shelling steps (relative to $D$) which deform $C$ to $C'$. 

The intermediate complexes between $C$ and $C'$ form a sequence of subcomplexes $(C=C_1,\cdots,C_n=C')$, which we call \Defn{shelling sequence relative to $D$}. This is in contrast to the \Defn{shelling order relative to $D$}, which is a list $(F_1, F_{2}, \cdots, F_n)$ of facets of $C$ in their order of removal from $C$.

If $D=\varnothing$, we abbreviate this to $C \searrow_{\S} C'$, and simply say $C$ \Defn{shells} to $C'$.  A complex $C$ is \Defn{shellable} if and only if $C$ shells to some complex $\widecheck{C}$ with $d (f_d(\widecheck{C})-1) = 0$. We will write this simply as \[C \searrow^\varnothing_{\S} \varnothing\quad \text{or even simpler} \quad C \searrow_{\S} \varnothing.\] 
\end{definition}

Note that shellings of simplicial complexes are defined in a much simpler way, due to the fact that any principal nonempty subcomplex $C$ of $\partial \Delta$, where $\Delta$ is any simplex, is shellable.

\begin{lemma} \label{lem:connsum} 
Let $B$ and $C$ be two pure polytopal $d$-complexes that intersect in a complex $D$, and let $D'\subset C$ be a complex containing $D$. The following are equivalent:
\begin{compactenum}[(a)]
\item $C\searrow^{D'}_{\S}C'$ via the shelling sequence $(C_i)$.
\item  $B \cup C \searrow^{D'\td D}_{\S} B\cup C'$ via the shelling sequence $(B\cup C_i)$. \qed
\end{compactenum} 
\end{lemma}

Hence, the notion of relative shellings gives us a way to glue shellable complexes to obtain larger shellable complexes. This is not trivial, and not always possible: For example, Rudin's ball is the union of two shellable $3$-balls $B$ and $C$, glued together at a (shellable) $2$-ball $D$ in their boundary. However, neither $B$ nor $C$ are shellable relative to $D$, since Rudin's ball is not shellable.

\begin{lemma} \label{lem:cone}
Let $C$ be simplicial complex with a shelling sequence $(C_i)$ associated to a shelling relative to a complex $D$. Then $(v\ast C_i)$ is a shelling sequence relative to $v\ast D$. \qed
\end{lemma}

To generalize this, let us say that if $(C_i)$ is shelling sequence relative to a complex $D$, then the shelling \Defn{restricts} to a principal subcomplex $E$ of $C_1$ if $(E_i=C_i\cap E)$ is a shelling sequence relative to~$D$.

\begin{lemma} \label{lem:cone2}
Let $C$ be simplicial complex with a shelling $(C_i)$ that restricts to a subcomplex $E\subset C$. Then $(v\ast C_i)$ is a shelling sequence of $v\ast C$ relative to $E\cup (v\ast D)$. \qed
\end{lemma}

If $C$ is a simplicial complex, and $v$ is one of its vertices, define $C \tu v$ as the maximal subcomplex of $C$ with $(C\tu \tau)-\tau= C\td \tau$. Then $\Lk(\tau, C\tu\tau)\cong \St(\tau,C)\cap C\td D$.
Combining Lemmas~\ref{lem:cone2} and~\ref{lem:connsum}, we obtain the following corollary:

\begin{cor}\label{cor:linkshell}
If $C$ is a simplicial complex, $v$ is a vertex of $C$, and if $\Lk(v,C)$ has a shelling relative to a complex $D$ that restricts to ${\Lk(v,C \tu v)}$, then $C\searrow^{v\ast D}_\S C\td v$. In particular, if $\Lk(v,C)\searrow^D_\S \Lk(v,C \tu v)$, and the latter complex is shellable, then $C\searrow^{v\ast D}_\S C\td v$. 
\end{cor}

\paragraph*{Reverse shellings and duality.} To illustrate the notion of relative shellings further, we notice the following simple fact.

\begin{prp}
Let $C$ denote a simplical ball with a shelling sequence that is simultaneously a shelling relative to $\varnothing$ and $\parti  C$. Then $\parti  C$ is shellable.
\end{prp}
 \begin{proof}
Let $(C_i)$ denote the shelling sequence that is a shelling and a $\parti  C$ shelling. Then $(C_i\cap \parti C)$ is a sequence of subcomplexes of $\parti C$. If elements of this sequence appear more than once, we delete their later occurrences from the sequence. We are left with a sequence $(C'_i)$ of simplicial subcomplexes of $\parti C$ which can be refined to a shelling sequence for $\parti C$, for instance by applying Lemma \ref{lem:bary} below.
\end{proof}

As a consequence we observe, once again, that a shelling of $C$ is not automatically a shelling relative to $\parti  C$ (and vice versa). 
\begin{example}
Let $S$ be any PL sphere that is not shellable, for instance one of the spheres of Lickorish~\cite{LME}. Then $S$ is the boundary of some shellable ball $B$ by Pachner's Theorem \cite{Pachner}. By the previous proposition, any shelling of $B$ is not a shelling relative to $\parti  B=S$.
\end{example}

This does not mean that a shellable ball $B$ does not admit a shelling relative to $\parti  B$, it might just not be the same shelling, as we shall see below. 

\begin{prp}[Alexander duality for relative shellings]\label{prp:rev}
Let $C$ be a polytopal $d$-ball with boundary $\parti  C$, and consider a principal subcomplex $D$ of $\parti  C$. If $C$ is shellable relative to $D$, then $C$ is shellable relative to ${D':=\parti  C\td D}$.
\end{prp}

\begin{proof}
Let $n:=f_d(C)$ be the number of facets of $C$.
If $(F_1, F_{2}, \cdots, F_{n})$ is the shelling order for $C$ relative to $D$, then we claim that enumerating the facets by $G_i:= F_{n+1-i}$ gives the desired shelling order $(G_1, \cdots G_{n})$ of $C$ relative to~$D'$. We say $(G_i)$ is the shelling order of the \Defn{reverse shelling} to the shelling given by $(F_1, F_{2}, \cdots, F_{n})$.

Let us argue why $(G_1, \cdots G_{n})$ is indeed associated to a shelling relative to $D'$ let $G_i= F_{n+1-i}$ denote any facet of the shelling order $(G_i)$. Let $(C_i)$ denote the shelling sequence associated to the shelling with order $(F_i)$, and let $(C'_i)$ the shelling sequence associated with the shelling with order $(G_i)$, where $C_{n+1}=C'_{n+1}=\varnothing$. With this notation
\[F_{n+1-i}\cap (C_{n+2-1}\cup D)\ \ \text{and}\ \ \parti F_{n+1-i}\td \left(F_{n+1-i}\cap (C_{n+2-i}\cup D)\right),\ \ 1\le i \le n,\] are pure $(d-1)$-dimensional and shellable by assumption. Since \[G_i\cap (C'_{i+1}\cup D')=\parti G_i-\left(G_i \cap (C_{n+2-i}\cup D)\right)=\parti F_{n+1-i}\td \left(F_{n+1-i}\cap (C_{n+2-i}\cup D)\right) \]
and similarly \[\parti G_i\td \left(G_i\cap (C'_{i+1}\cup D')\right)=F_{n+1-i}\cap (C_{n+2-i}\cup D),\] the removal of $G_i$ from $C'_i$ is a valid shelling step relative to $D'$.
\end{proof}

\begin{cor}\label{cor:eq}
Let $C$ denote a PL sphere, and let $A$ and $B=C\td A$ denote shellable subcomplexes of $C$. Then $A$ is the ``beginning segment'' of a shelling of $C$, i.e. \[C\searrow_\S C\td A \searrow_\S \varnothing. \qedhere\]
\end{cor}

This shows the equivalence of our notion of shellability and the one classically used, cf.\ \cite{BT, Z}.

\begin{rem}
Topologically, the scope of Proposition~\ref{prp:rev} is rather limited: If $C$ is a polytopal ball, and $C\searrow_\S^D \varnothing$, where $D$ is a principal subcomplex of $\parti C$, then either $D=\parti  C,\ D=\varnothing$ or $D$ is a PL ball in $\parti  C$. 
\end{rem}

Finally, we recall some classical results on shellings. If $C$ is a polytopal complex, and $\tau$ is not an element of $C$, then we define the stellar subdivision of $C$ at $\tau$ as $\st ({\tau,C})=C$.

\begin{lemma}[cf.\ {\v{C}}uki{\'c}--Delucchi \cite{CD}]\label{lem:bary}
If $C$ is a simplicial complex that shells to subcomplex $C'$ relative to a complex $D$. Then any stellar subdivision $\st ({\tau,C})$ of $C$ shells to the corresponding stellar subdivision  $\st ({\tau,C'})$ relative to $\st ({\tau,D})$.\qed
\end{lemma}

It is not hard to see that shellability is also preserved under passing to links.
\begin{prp}[cf.\ Lem.\ 8.7 \cite{Z}]
If $C\searrow_\S^D C'$ for any complexes $C$, $C'$ and $D$, and $\tau$ is any face of $C$, then $\Lk(\tau,C)\searrow_\S^{\Lk(\tau,D)} \Lk(\tau,C')$. \qed
\end{prp}

As consequence of Lemma~\ref{lem:bary}, we see that shellability is preserved under derived subdivisions. Next, we present a lemma for joins of shellable complexes.

\begin{lemma}[cf.\ Bj\"orner--Wachs \cite{BW}]\label{lem:jshell}
Given any two simplicial complexes $A$, $A'$, $B$, $B'$ and $D$, the following are equivalent:
\begin{compactenum}[(1)]
\item $A\searrow_\S A'$ and $B\searrow_\S^D B'$
\item $A \ast B\searrow_\S^{A\ast D} A'\ast B'$.
\end{compactenum}
\end{lemma}

\begin{proof}
We will only treat the case $(1) \Rightarrow (2)$ here, as it will be useful later to see how the desired shelling sequence is obtained. For the case $(2) \Rightarrow (1)$, we refer the reader to \cite{BW}. 

Now, if $A_1, A_2, \cdots A_j$ is the associated shelling sequence for $A$, and $B_1, B_2, \cdots B_i$ is the associated shelling sequence for $B$, then 
\begin{align*}
\big(A\ast B{=}&A_1\ast B_1{=}(A_1\ast B_2)\cup (A_1\ast (B_2\td B_1)),\cdots, (A_1\ast B_2)\cup (A_j\ast (B_2\td B_1)),\\ 
&A_1\ast B_2{=}(A_2\ast B_3)\cup (A_1\ast (B_3\td B_2)),\cdots , (A_1\ast B_3)\cup (A_j\ast (B_3\td B_2)), \\ 
&\qquad\qquad\qquad\qquad\vdots \qquad\qquad\qquad\qquad \vdots\qquad\qquad\qquad\qquad \vdots \\
&\qquad\qquad\qquad\qquad\vdots \qquad\qquad\qquad\qquad \vdots \qquad\qquad\qquad\qquad \vdots  \\
&A_1\ast B_{i-1}{=}(A_1\ast B_i)\cup (A_1\ast (B_{i-1}\td B_{i})),\cdots, (A_j\ast B_{i})\cup (A_j\ast (B_{i-1}\td B_{i})),A_1\ast B_{i}{=}A\ast B' \big)
\end{align*}
shells $A \ast B$ to $A \ast B'$ relative to $A\ast D$. Analogously, we can continue the shelling of $A \ast B'$ to $A' \ast B'$.
\end{proof}

Finally, recall that a \Defn{path} of $d$-simplices is a simplicial complex homeomorphic to a ball whose dual graph is a path. 

\begin{lemma}\label{lem:path}
A path of $d$-simplices is shellable to any of its facets.
\end{lemma}

\begin{proof}
Iteratively remove the leaves of the path; since up to the last step there are always at least $2$ leaves, we can choose to leave any chosen facet untouched by the shelling.
\end{proof}

\paragraph*{Derived orders and derived neighborhoods.} We recall two notions related to derived subdivisions; the first is taken from \cite{AD2013}, the second is classical.

\begin{definition}[Derived order]\label{def:extord}
An \Defn{extension} of a partial order $\prec$ on a set $S$ is any partial order $\widetilde{\prec}$ on any superset $T$ of $S$ such that $a\, \widetilde{\prec}\, b$ whenever $a \prec b$ for $a,b\in S$.

Let now $C$ be a polytopal complex, let $S$ denote a subset of its faces, and let $\prec$ denote any strict total order on $S$ with the property that $a\prec b$ whenever $b\subset a$. We extend this order to an irreflexive partial order $\widetilde{\prec}$ on $C$ as follows: Let $\sigma$ be any face of $C$, and let $\tau\subsetneq \sigma$ be any strict face of $\sigma$. 
\begin{compactitem}[$\circ$]
\item If $\tau$ is the minimal face of $\sigma$ under $\prec$, then $\tau\,\widetilde{\prec}\, \sigma$.
\item If $\tau$ is any other face of $\sigma$, then $\sigma\, \widetilde{\prec}\, \tau$.
\end{compactitem}
The transitive closure of the relation $\widetilde{\prec}$ gives an irreflexive partial order on the faces of $C$, and by the correspondence of faces of $C$ to the vertices of $\sd\sm C$, it gives an irreflexive partial order on $\F_0(\sd\sm C)$. Any strict total order that extends the latter order is a \Defn{derived order} of $\F_0(\sd\sm C)$ induced by $\prec$.
\end{definition}

\begin{definition}[Derived neighborhoods, cf.\ \cite{ZeemanBK, Hudson}]
Let $C$ be a polytopal complex. Let $D$ be a set of faces of $C$ (for instance a subcomplex of $C$). The \Defn{(first) derived neighborhood} $N(D,C)$ of $D$ in $C$ is the polytopal complex
\[N(D,C):=\bigcup_{\sigma\in \sd\sm D} \St(\sigma,\sd\sm C).\]
\end{definition}

\begin{prp}\label{prp:js}
Let $K$ and $\varLambda$ denote simplicial complexes, and let $V$ denote a subset of the vertex set of $\varLambda$. Assume $\sd\sm K$ is shellable relative to a subcomplex $\sd\sm D$, $D\subset K$, and that \[\sd\sm \varLambda\searrow_\S N(V,\varLambda)\searrow_\S \varnothing.\] Then $\sd\sm (K\ast \varLambda)$ has a shelling relative to $\sd\sm (D\ast \varLambda)$ that restricts to ${N(K\ast V, K\ast \varLambda)}$.
\end{prp}

\begin{proof} We abbreviate $X:=\sd\sm K \ast \sd\sm \varLambda$, $\varUpsilon:=\sd\sm (K \ast \varLambda )$, $\chi:=\sd\sm D\ast \sd\sm \varLambda$ and $\upsilon:=\sd\sm (D\ast   \varLambda)$. 
First, notice that by Lemma~\ref{lem:jshell} \[\sd\sm K \ast \sd\sm \varLambda=X \searrow^{\chi}_\S\sd\sm K \ast N(V,\varLambda)\searrow^{\chi}_\S \varnothing.\]
Let $e$ be any edge of $X$. If the vertices of $e$ correspond to faces of dimension $k$ of $K$ and $\lambda$ of $\varLambda$, respectively, then let us define the \Defn{height} of $e$ as $k+\lambda$. Perform stellar subdivisions at the edges of $\sd\sm K \ast \sd\sm \varLambda$ not in $\sd\sm K \cup \sd\sm \varLambda$ in order of decreasing height of the edges. The resulting complex is combinatorially equivalent to $\sd\sm (K \ast \varLambda)$ and is therefore identified with the latter.

Since stellar subdivisions preserve shellability (Lemma~\ref{lem:bary}) we obtain that whenever \[X=A_1, A_2, \cdots, A_j , \sd\sm K \ast N(V,\varLambda)=B_1,B_2,\cdots \] is a shelling sequence (relative to $\chi$) for $X$, then 
\[\RS(\varUpsilon,A_1),\RS(\varUpsilon,A_2),\cdots, \RS(\varUpsilon,A_j), \RS(\varUpsilon,B_1), \RS(\varUpsilon,B_2),\cdots\]
can be refined to a shelling sequence (relative to $\upsilon$) of $\varUpsilon$. It remains to describe such a shelling that restricts to ${N(K\ast V, K\ast \varLambda)}$. This can be done as follows: 
Set $\varDelta_i:= A_i \td A_{i+1}$. Then $\RS(\varUpsilon,\varDelta_i)$ is a path of $m:=(\dim K+1)\cdot(\dim \varLambda+1)$ facets of $\varUpsilon$. Using Lemma~\ref{lem:path}, shell $\RS(\varUpsilon,\varDelta_i)$ to the unique facet of $\RS(\varUpsilon,\varDelta_i)$ containing $\varDelta_i \cap \sd\sm K$, and let $\varDelta_{i,1}, \varDelta_{i,2}, \cdots, \varDelta_{i,m}$ denote the associated shelling sequence. Then we shell 
$\varUpsilon$ relative to $\upsilon$ using the shelling sequence
\begin{align*}
\big(\varUpsilon=&\RS(\varUpsilon,A_1)=\RS(\varUpsilon,A_2)\cup\varDelta_{1,1},\RS(\varUpsilon,A_2)\cup\varDelta_{1,2},\cdots, \RS(\varUpsilon,A_2)\cup\varDelta_{1,m},\\  
&\RS(\varUpsilon,A_2)=\RS(\varUpsilon,A_3)\cup\varDelta_{2,1},\RS(\varUpsilon,A_3)\cup\varDelta_{2,2},\cdots, \RS(\varUpsilon,A_3)\cup\varDelta_{2,m},\\
&\qquad\qquad\qquad\vdots\qquad\qquad\ \qquad \vdots\qquad\qquad\ \qquad \vdots\\
&\qquad\qquad\qquad\vdots\qquad\qquad\ \qquad \vdots\qquad\qquad\ \qquad \vdots\\
&\RS(\varUpsilon,A_j)=\RS(\varUpsilon,B_1)\cup\varDelta_{j,1},\RS(\varUpsilon,B_1)\cup\varDelta_{j,2},\cdots, \RS(\varUpsilon,B_1)\cup\varDelta_{j,m},\RS(\varUpsilon,B_1) \big) 
\end{align*}
This shells $\varUpsilon$ to $\RS(\varUpsilon,B_1)$ relative to $\upsilon$. It remains to show that $\RS(\varUpsilon,B_1)$ is shellable relative to $\upsilon$; whatever shelling we choose for this will automatically restrict to ${N(K\ast V, K\ast \varLambda)}$ since $\RS(\varUpsilon,B_1)\subset {N(K\ast V, K\ast \varLambda)}$. Hence, we can finish the proof by applying Lemma~\ref{lem:bary}, since $B_1$ is shellable relative to $\chi$.
\end{proof}

\subsection{Shellability of convex balls}

We can now prove Main Theorem~\ref{mainthm:ShellableSub}. We first prove that any subdivision of a convex polytope (of arbitrary dimension) is shellable after $2$ derived subdivisions; afterwards, we argue that $1$ derived subdivision suffices for $d\le 3$. The essence of the proof can be summarized as follows: 

\begin{compactitem}[$\circ$]
\item For general $d$, we start with a derived subdivision in order to be able to order the vertices nicely; for $d \le 3$ this step turns out not to be necessary. Let us call the resulting complex $X$.
\item We now order the vertices $v_1,\cdots, v_n$ of $X$ by an order coming from the geometry of the complex. Then, we define the complexes $\varSigma_i:=\sd\sm X\td\{v_0, v_1, \, \cdots, v_{i-1}\}$ and prove the shellability of $\sd\sm X$ by proving \[\varSigma_i\searrow^D_\S \varSigma_{i}\tu v_i \searrow^D_\S \varSigma_{i}\td v_i=\varSigma_{i+1}\]
for every $i$, compare Figure \ref{fig:shelling}.
\end{compactitem}

\begin{figure}[htbf]
  \centering 
  \includegraphics[width=0.9\linewidth]{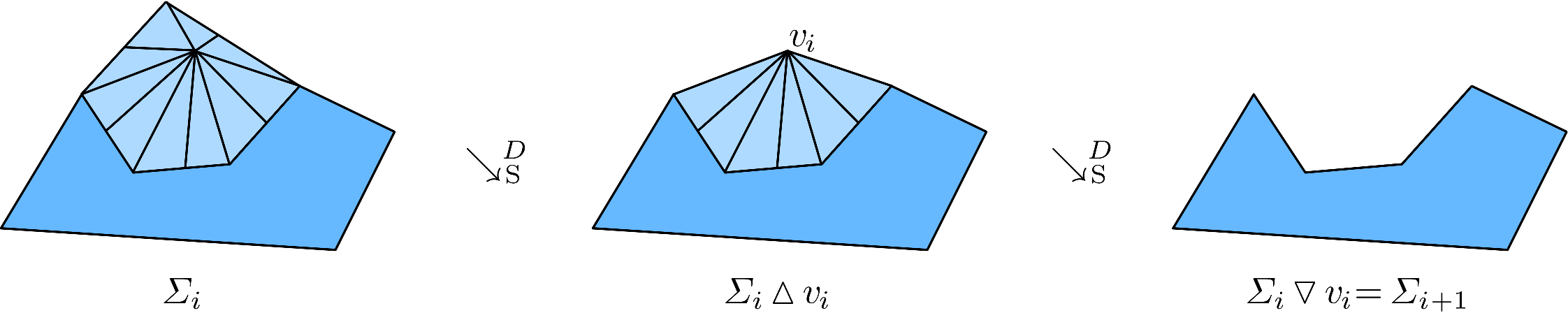} 
\caption{Illustration for the shelling procedure.}
  \label{fig:shelling}
\end{figure}

\paragraph*{Arbitrary dimension.}

A hemisphere is in \Defn{general position} w.r.t.\ a complex $C$ in $S^d$ if its boundary contains no vertex of $C$. Recall that a \Defn{polyhedron} in $S^d$ is a subset of $S^d$ that is obtained as the intersection of a finite number of closed halfspaces in the sphere.

\begin{theorem}\label{thm:shs} 
Let $C$ be any subdivision of a polyhedron $P$ in $S^d$. Let $H_+$ denote a general position hemisphere in $S^d$ that intersects $C$ non-trivially. Let ${H_-}$ denote the hemisphere complementary to $H_+$. Define $D=\varnothing$ or $D=\sd^2\parti  C$. Then
\begin{compactenum}[\rm (I)]
\item $\sd\sm N(\RS(C,H_+),C)$ is shellable relative to $D$, and
\item $\sd^2 C\searrow_\S^{D}\sd\sm N(\RS(C,H_+),C)\searrow_\S^{D}\varnothing$. 
\end{compactenum} 
\end{theorem}

\begin{cor}\label{cor:s}
Let $C$ be any subdivision of a convex polytope in $\R^d$. Then $\sd^2 C$ is shellable.
\end{cor}

Before we begin with the proof, let us state a variant of Theorem~\ref{thm:shs} that will become useful later. The proof is analogous to the proof of Theorem~\ref{thm:shs}, and left out.

\begin{theorem}\label{thm:shs2} 
Let $C$ be any subdivision of a polyhedron $P$ in $S^d$. Let $p$ denote any facet of $P$ and let $H_+$ denote a general position hemisphere in $S^d$ that intersects $P$ and $p$ non-trivially. Let ${H_-}$ denote the hemisphere complementary to $H_+$. Define $D:=\RS(\sd^2 C, p)$ or $D:=\RS(\sd^2 C, \parti P{\sm\setminus\sm} \rint p)$. Then 
\begin{compactenum}[\rm (I)]
\item $\sd\sm N(\RS(C,H_+),C)$ is shellable relative to $D$, and 
\item $\sd^2 C\searrow_\S^{D}\sd\sm N(\RS(C,H_+),C)\searrow_\S^{D}\varnothing$. 
\end{compactenum} 
\end{theorem}

\begin{proof}[\textbf{Proof of Theorem~\ref{thm:shs}}]

The proof is by induction on the dimension; the case $d=1$ is obvious. We will see that for a fixed dimension $d$, {(I)} implies {(II)} straightforwardly. Hence, we first prove {(I)} for dimension $d+1$, and we may assume that {(I)} and {(II)} have been already proven for all dimensions lower than $d$. 

\paragraph*{Proof of (I).} Let $x$ denote the midpoint of ${H_+}$, and let $\mathrm{d}(y)$ denote the distance of a point $y \in S^d$ to $x$ with respect to the canonical metric on $S^d$. Let $\mathrm{M}(C, {H_+})$ denote the set of faces $\sigma$ of $\RS(C,l)$ for which the function ${\arg\min}_{y\in \sigma} \mathrm{d}(y)$ attains its minimum in the relative interior of $\sigma$. With this, we order the elements of $\mathrm{M}(C, {H_+})$ strictly by defining $\sigma\prec \sigma'$ whenever $\min_{y\in\sigma}\mathrm{d}(y)<\min_{y\in\sigma'}\mathrm{d}(y)$. Since ${H_+}$ is in general position, we may perturb it slightly without changing $N(\RS(C,H_+),C)$, and therefore we may assume that $\prec$ is a strict total order.

This allows us to induce an associated derived order on the vertices of $\sd\sm C$, which we restrict to the vertices of $N(\RS(C,L),C)$. Let $v_0, v_1, v_2, \, \cdots, v_n$, $n=f_0\left(N(\RS(C,L),C)\right)$, denote the vertices of $N(\RS(C,L),C)$ labeled according to the latter order, starting with the maximal element $v_0$. Let $\varSigma_i$ denote the complex $\sd\sm N(\RS(C,L),C)\td\{v_0, v_1, \, \cdots, v_{i-1}\}$. We will prove that $\varSigma_i\searrow^D_\S \varSigma_{i+1}$ for all $i$, $0\le i\le n-1$. 
To finish the proof, we then only have to observe that $\varSigma_n$ is a join of $v_n$ and $\Lk(v_0,\sd^2 C)$, which is shellable relative to $\Lk(v,D)$ by inductive assumption. Hence, by Lemma~\ref{lem:jshell}, $\varSigma_n$ is shellable relative to $D$. To prove  $\varSigma_i\searrow^D_\S \varSigma_{i+1}=\varSigma_i\td v_i$, there are two cases to consider.

\begin{compactenum}[\em (1)]
\item $v_i$ corresponds to an element of $\mathrm{M}(C,L)$.
\item $v_i$ corresponds to a face of $C$ not in $\mathrm{M}(C,L)$.
\end{compactenum}	
\noindent We need some notation to prove these two cases. Recall that we can define $\RN$, $\RN^1$ and $\Lk$ with respect to a base-point; we shall need this notation in case \emph{(1)}. We shall abbreviate $v:=v_i$ for the duration of the proof. Let us denote by $\tau$ the face of $C$ corresponding to $v$ in $\sd\sm C$, and let $m$ denote the point ${\arg\min}_{y\in \tau} \mathrm{d}(y)$. Finally, define the ball $B_m$ as the set of points $y$ in $S^d$ with~$\mathrm{d}(y)\le\mathrm{d}(m)$, and define $D':=\Lk(v,D)$.

\medskip
\noindent \emph{Case (1)}: 
 In this case, \[\Lk(v,\varSigma_i)\cong \sd\sm (\sd\sm \parti  \tau \ast \sd\sm \Lk(\tau,C))\] and $\Lk(v,\varSigma_i\tu v)\subset\Lk(v,\varSigma_i)$ is combinatorially equivalent to \[\sd\sm \big(\sd\sm \parti  \tau \ast N (\mathrm{LLk}_m(\tau,C), \Lk(\tau,C))\big),\ \text{where}\ \mathrm{LLk}_m(\tau,\sigma):=\RS(\Lk(\tau,C),\RN^1_{(m,\tau)} B_m).\] 
The complexes $\sd\sm (\sd\sm \parti  \tau \ast \sd\sm \Lk(\tau,C))$ and $\sd\sm \big(\sd\sm \parti  \tau \ast N (\mathrm{LLk}_m(\tau,C), \Lk(\tau,C))\big)$ are obtained as stellar subdivisions of the complex $\sd^2 \parti  \tau \ast \sd^2 \Lk(\tau,C)$ and the complex \[\sd^2 \parti  \tau \ast \sd\sm N (\mathrm{LLk}_m(\tau,C), \Lk(\tau,C))\subset \sd^2 \parti  \tau \ast \sd^2 \Lk(\tau,C),\] respectively, compare also the proof of Proposition~\ref{prp:js}.
Hence, by inductive assumption {(II)} and Lemma~\ref{lem:bary}, $\Lk(v,\varSigma_i)\searrow^{D'}_\S\Lk(v,\varSigma_i\tu v)\searrow^{D'}_\S \varnothing$. Hence, Corollary~\ref{cor:linkshell}  proves \[\varSigma_i\searrow^D_\S \varSigma_{i}\tu v \searrow^D_\S \varSigma_{i}\td v=\varSigma_{i+1}.\]

\smallskip
\noindent \emph{Case (2)}: If $\tau$ is not an element of $\mathrm{M}(C,H_+)$, let $\sigma$ denote the face of $\tau$ containing $m$ in its relative interior. Then \[\Lk(v, \varSigma_i)\cong \sd\sm (\sd\sm \parti  \tau \ast \sd\sm \Lk(\tau,C))\] and \[\Lk(v,\varSigma_i\tu v)\cong N\left(v_\sigma\ast \sd\sm \Lk(\tau,C),\sd\sm \parti  \tau \ast \sd\sm \Lk(\tau,C)\right),\] where $v_\sigma$ is the vertex of $\sd\sm \parti  \tau$ corresponding to $\sigma$.
By the inductive assumption (II),  $\sd^2  \Lk(\tau,C)=\sd\sm (\sd\sm \Lk(\tau,C))$ is shellable relative to $\sd^2  \Lk(\tau,D)$. Furthermore, $\sd^2 \parti  \tau$ can be realized as boundary of a polytope since $\tau$ is a polytope, and therefore $\sd^2 \parti  \tau$ shells to the shellable complex $\St(v_\sigma, \sd^2 \tau)=N(v_\sigma, \sd\sm \parti  \tau)$ by the Bruggesser-Mani Theorem \cite{BruggesserMani}. Hence, Proposition~\ref{prp:js} shows that
that $\Lk(v, \varSigma_i)$ is shellable relative to $D'=\sd\sm (\sd\parti  \tau \ast \sd\sm \Lk(\tau,C))$, and this shelling can be chosen to restrict to ${\Lk(v,\varSigma_i\tu v)}$.
Corollary~\ref{cor:linkshell} now proves \[\varSigma_i\searrow^D_\S \varSigma_{i}\tu v\searrow^D_\S\varSigma_{i}\td v = \varSigma_{i+1}.\]
This finishes the proof of {(I)}.

\paragraph*{Proof of (II).} It suffices to prove the case where $D=\varnothing$, the case $D=\parti  \sd^2 C$ is then obtained by reversing the shelling. Observe that  \[A\cup B=\sd^2 C\ \text{and}\ A\cap B=\parti  A\td \parti  \sd^2 C=\parti  B \td\parti  \sd^2 C,\] where $A:=\sd\sm N(\RS(C,H_+),C)$ and $B:=\sd\sm N(\RS(C,{H_-}),C)$.  By assumption on {(I)} the complex $B$ is shellable relative to $\parti  \sd^2 C$. Then $B$ is shellable relative to $\parti  B=A\cap B$ by Proposition~\ref{prp:rev}. Hence by Lemma~\ref{lem:connsum}  we see that $\sd^2 C\searrow_{\S} A$, and $A$ is shellable by {(I)}. 
\end{proof}

\paragraph*{Dimension $d\le 3$.}
To finish the proof of Theorem~\ref{mainthm:ShellableSub}, it remains to show that for $d\le 3$, one derived subdivision is enough.

\begin{theorem}\label{thm:sh3} 
Let $C$ be any subdivision of a polytope in $\R^d$, $d\le 3$. Then $\sd\sm C$ is shellable.
\end{theorem}

The proof relies on the following analogue of Theorem~\ref{thm:shs} for $d\le 2$:

\begin{prp}\label{prp:shs} 
Let $C$ be any subdivision of a polyhedron in $S^d$, $d\le 2$. Let $H_+$ denote a general position hemisphere in $S^d$ that intersects $C$ non-trivially. Let ${H_-}$ denote the hemisphere complementary to $H_+$. Define $D=\varnothing$ or $D=\sd\sm \parti  C$. Then
\begin{compactenum}[(I)]
\item $ N(\RS(C,H_+),C)$ is shellable relative to $D$, and
\item $\sd\sm C\searrow_\S^D N(\RS(C,{H_+}),C)\searrow_\S^D\varnothing$.
\end{compactenum} 
\end{prp}

\begin{proof}
We already saw that (I) straightforwardly implies (II) in the proof of Theorem~\ref{thm:shs}. Hence, it remains to show that 
$X:=N(\RS(C,H_+),C)\searrow^D_\S \varnothing$. The case $D=\varnothing$ (and therefore also the case $D=\parti X$) is folklore. 
To prove the case $\varnothing\subsetneq D\subsetneq \parti X$, pick any facet of the simplicial complex $X$ to find a facet $\varDelta$ for which $\varDelta\cap (D - \parti  X)$ is $1$-dimensional. 

\begin{figure}[htbf]
  \centering 
  \includegraphics[width=0.64\linewidth]{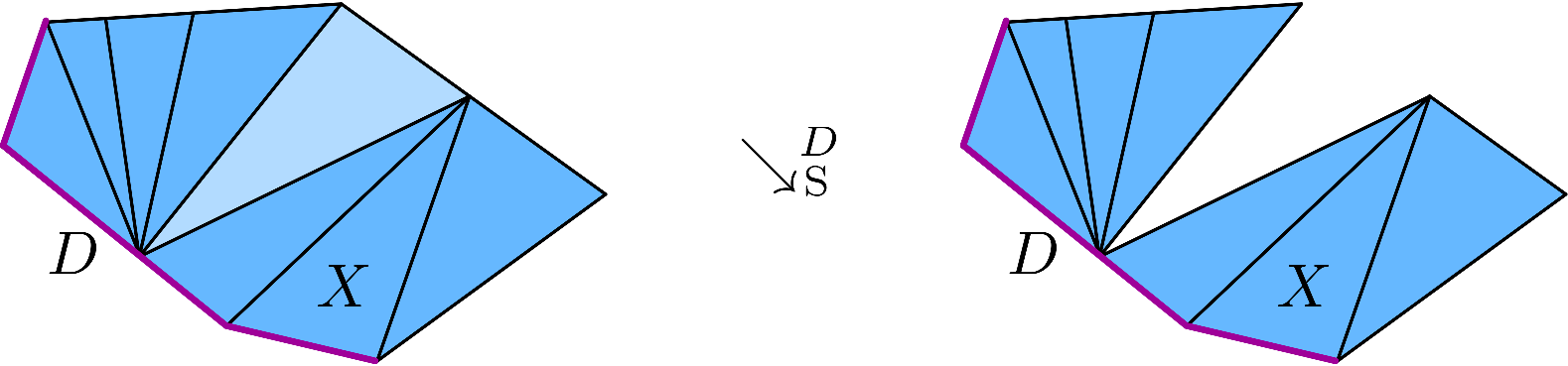} 
\caption{If $X$ is any planar pure simplicial complex, $D$ is a principal subcomplex of $\partial X$ and $\varDelta$ is any facet intersecting $\partial X \td D$ in a $1$-dimensional set, then $X\searrow_S^D X\td \varDelta$.}
  \label{fig:plshell}
\end{figure}

Then $X\rightarrow X\td \varDelta$ is a shelling step relative to $D$, cf\ Figure \ref{fig:plshell}. Proceeding this way, all facets of $X$ can be removed with shelling steps relative to $D$.
\end{proof}

\begin{proof}[\textbf{Proof of Theorem~\ref{thm:sh3}}]
Let $\mu$ be a generic vector in $\R^d$, and order the vertices of $C\subset \R^d$ according to their value under $\langle\mu,\cdot\rangle$. Since $\mu$ is generic, this gives a total order on the vertices of $C$. Let $v_0, v_1, v_2, \, \cdots, v_n$, $n=f_0(C)$, denote the vertices of $C$ labeled according to that order, starting with the vertex $v_0$ maximizing the interior product $\langle\mu,\cdot\rangle$. 
Let $\varSigma_i$ denote the complex $C\td \{v_0, v_1, \, \cdots, v_{i-1}\}$. We will prove that $\varSigma_i\searrow_\S \varSigma_{i+1}$ for all $i$, $0\le i\le n-1$. 

Set $v:=v_i$ and define $B_v$ as the set of points $y$ in $S^d$ with~$\mathrm{d}(y)\le\mathrm{d}(v)$. Then $\Lk(v,\varSigma_i)\cong  \sd\sm \Lk(\tau,C)$, and \[ \Lk(v,\varSigma_i\tu v)\cong N (\mathrm{LLk}_v(\tau,C), \Lk(\tau,C)),\ \text{where}\ \mathrm{LLk}_v(\tau,\sigma):=\RS(\Lk(\tau,C),\TT^1_{v} B_v).\] 
Hence, by Proposition~\ref{prp:shs}, \[\Lk(v,\varSigma_i)\searrow_\S\Lk(v,\varSigma_i\tu v)\searrow_\S \varnothing.\] Thus, Corollary~\ref{cor:linkshell}  proves \[\varSigma_i\searrow_\S \varSigma_{i}\tu v \searrow_\S \varSigma_{i}\td v=\varSigma_{i+1}.\]
To finish the proof, observe that $\varSigma_n$ is a join of $v_n$ and $\Lk(v_0,\sd^2 C)$, which is shellable relative to $\Lk(v,D)$ by Proposition~\ref{prp:shs}. Hence, by Lemma~\ref{lem:jshell}, $\varSigma_n$ is shellable. 
\end{proof}

\subsection{Consequences in combinatorial topology}

We derive Theorem~\ref{mainthm:equivalence} and some further consequences in PL topology.

\paragraph*{Some basic facts in simplicial approximation.}

Let us start by recalling a fundamental approximation Lemma of PL topology:

\begin{lemma}[cf.\ {\cite[Ch.\ 1, Lem. 4 ]{ZeemanBK}}]\label{lem:common}
Let $C$ be a polytopal complex, and let $C'$ be any subdivision of $C$. Then there is an $n\ge 0$ such that some $n$-th derived subdivision $\sd^n C$ of $C$ is a subdivision of both $C$ and~$C'$.
\end{lemma}

\begin{cor}\label{cor:PLhcommon}
Let $C$ and $D$ be two PL homeomorphic polytopal complexes. Then there is an $n\ge 0$ and a subdivision $D'$ of $D$ such that $\sd^n C\cong D'$.
 \end{cor}

The second result we shall need is a little more delicate; according to Zeeman, it goes back at least to Alexander and Newman. 

\begin{theorem}[{\cite[Ch.\ 3, Thm.\ 2]{ZeemanBK}}]
Let $C_1$ and $C_2$ denote two PL $d$-balls, and let $D_1$, $D_2$ denote PL $(d-1)$-balls in their respective boundaries. Then every PL homeomorphism $\phi: D_1\mapsto D_2$ extends to a PL homeomorphism $\phi: C_1\mapsto C_2$.
\end{theorem}

If we combine this with Lemma~\ref{lem:common}, we obtain:

\begin{lemma}\label{lem:bb}
Let $C$ be any PL $d$-ball, and let $D$ denote any PL $(d-1)$-ball in $\parti C$. Let $\varDelta$ be the $d$-simplex, and let $\delta$ be any one of its facets. Then there is an $n\ge 0$ and a subdivision $\varDelta'$ of $\varDelta$ such that $\sd^n D=\RS(\sd^n C,D)\cong \delta':=\RS(\varDelta',\delta)$, and the combinatorial isomorphism extends to an isomorphism $\sd^n C\mapsto \Delta'$.
\end{lemma}

Finally, we state a result that allows us to replace PL homeomorphisms with subdivisions when dealing with shellability.

\begin{prp}\label{prp:nohm}
Let $C$ and $D\subset C$ denote polytopal complexes PL homeomorphic to complexes $\varGamma=\phi(C)$ and $\varDelta=\phi(D)=\RS(\varGamma,\phi(D))$. Assume that $\varGamma$ shells to $\varDelta$. Then there exist subdivisions $C'$ of $C$ and $D'=\RS(C',D)$ of $C$ and $D$ such that $C'$ shells to $D'$.
\end{prp}

\begin{proof}
By Lemma~\ref{lem:common}, there is an $n\ge 0$ and a subdivision $C'$ of $C$ such that $\sd^n \varGamma\cong C'$, and such that $\sd^n \varDelta$ gets taken to $D':=\RS(C',D)$ under this combinatorial isomorphism. Since barycentric subdivisions preserve shellability (Lemma~\ref{lem:bary}), we see that $\sd^n \varGamma\searrow_\S \sd^n \varDelta$ and therefore $C'\searrow_\S D'$.  
\end{proof}

\paragraph*{Characterization of shellability.}

As a direct consequence of Lemma~\ref{lem:bb} and Theorems~\ref{thm:shs} and~\ref{thm:shs2}, we note the following:

\begin{theorem}\label{thm:ttb}
Let $C$ be any PL ball, and let $D=\parti C$, $D=\varnothing$ or $D\subset\parti C$ a PL $(d-1)$-ball. Then there is an $n\ge 0$ so that $\sd^n C\searrow_\S^{\sd^n D}\varnothing$.  
\end{theorem}

We conclude that shellability is almost preserved under subdivisions.

\begin{theorem}\label{thm:HCS}
Let $C$ and $D\subset C$ be polytopal complexes such that $C$ shells to $D$. Let $C'$ be any subdivision of~$C$. Then there is an $n\ge 0$ so that $\sd^{n} C'$ collapses to $\sd^{n} D'$ where $D'=\RS(C',D)$.
\end{theorem}

\begin{rem}
 By using a variant of Theorem \ref{thm:shs2} one can show that in the notation of Theorem \ref{thm:HCS}, it suffices to take $n=2$.
\end{rem}

\begin{proof}
We may assume that $C\searrow_\S D$ is only a single shelling step that consists of the removal of a single facet $\varDelta$. Set $\varDelta':=\RS(C', \varDelta)$. Now, $\varDelta \cap D$ is either a PL $(d-1)$-ball or coincides with $\parti \varDelta$. Hence, the same is true for $\varDelta'\cap D'$ which is a subdivision of $\varDelta \cap D$. By Theorem~\ref{thm:ttb}, there is an $n\ge 0$ such that $\sd^n \varDelta'$ is shellable relative to $\sd^n \varDelta'\cap \sd^n D'=\sd^n (\varDelta'\cap D')$. Hence, by Lemma~\ref{lem:connsum}, \[\sd^n C'=\sd^n \varDelta'\cup \sd^n D'=\sd^n (\varDelta'\cup D')\searrow_\S \sd^n D'. \qedhere\]
\end{proof}

Finally, we can finish the proof of Theorem~\ref{mainthm:equivalence}; by Proposition~\ref{prp:nohm}, it is equivalent to the following theorem, which we will prove in its stead.

\begin{theorem}\label{thm:shellyyo}
Let $C$  and $D\subset C$ be polytopal complexes such that $C \searrow_\S D$. Suppose that some subdivision $C'$ of $C$ shells onto some subdivision $D'$ of $D$. Then for $n$ large enough,
$\sd^{n} C  \searrow_\S  \sd^{n} D.$
\end{theorem}

\begin{proof}
By Lemma~\ref{lem:common}, there is an $\ell\ge 0$ so that $\sd^\ell C$ is a subdivision of $C'$, and such that $\sd^\ell D$ is a subdivision of $D'=\RS(C',D)$. Moreover, by Theorem~\ref{thm:HCS}, there is an $m\ge 0$ so that \[\sd^{m+\ell} C =\sd^m  \sd^\ell C \searrow_\S \sd^m  \sd^\ell D= \sd^{m+\ell} D.\qedhere\]
\end{proof}

Now, a \Defn{triangulated manifold} is a simplicial complex whose underlying space is a topological manifold.

\begin{cor} \label{cor:characterizationPL}
For a triangulated manifold $M$, the following are equivalent:
\begin{compactenum}[\rm (1) ]
\item $M$ is a PL ball or a PL sphere.
\item For some $m\ge 0$, the $m$-th derived subdivision of $M$ is shellable.
\end{compactenum}
\end{cor}

\begin{proof}
 \begin{compactitem}[$\circ$]
\item \makebox[4.4em][l]{(1) $\Rightarrow$ (2)} is well known: One can prove by induction on the dimension, combined with elementary cellular homology and Poincar\'e duality, that all shellable $d$-pseudomanifolds, for $d > 0$, are either PL balls or PL spheres, cf.\ \cite{DK}.  
\item \makebox[4.4em][l]{(2) $\Rightarrow$ (1)} follows from Theorem~\ref{thm:shellyyo}, and the fact that every PL $d$-ball has some subdivision which is also some iterated derived subdivision of the $d$-simplex (hence shellable).\qedhere
\end{compactitem}
\end{proof}

Corollary \ref{cor:characterizationPL} answers a problem of Billera and Swartz. Since shellability is closely connected to convex polytopes, the result is closely related to an important problem concerning realizations of convex polytopes.

\begin{problem}[Billera]\label{pr:bill}
Let $C$ be any PL sphere. Is there an $n\ge 0$ such that $\sd^n C$ is combinatorially equivalent to the boundary of a polytope? 
\end{problem}

We conclude by expanding Corollary~\ref{cor:characterizationPL} to a characterization of PL manifolds:

\begin{cor} \label{cor:characterizationPL2}
For any triangulated manifold $M$, the following are equivalent:
\begin{compactenum}[\rm (1) ]
\item $M$ is PL.
\item Every vertex link in $M$ becomes shellable after suitably many derived subdivisions.
\item For some $m\ge 0$, the $m$-th derived subdivision of any vertex link in $M$ is shellable.
\item For some $m \ge 0$, all vertex links in $\sd^m M$ are shellable.
\end{compactenum}
\end{cor}

\setcounter{theorem}{0}
\section{Collapsibility, products, and the Zeeman conjecture}
The work of the previous chapter shows that derived subdivisions induce shellability. An analogue of Theorem~\ref{thm:shellyyo}  can be proven for collapsibility, either by invoking Theorem \ref{thm:HCS}, or by using the more powerful result proven in the first author's PhD thesis \cite{AD2013}:

\begin{theorem}[{\cite[Thm.\ 2.3.B]{AD2013}}]\label{thm:HC}
Let $C, D$ be polytopal complexes such that $C$ collapses to $D$. Let $C'$ be any subdivision of $C$. Then $\sd\sm C'$ collapses to $\sd\sm D'$, where $D'=\RS(C',D)$.
\end{theorem}

This proves Theorem~\ref{mainthm:equivalencec}:
 
\begin{cor}\label{thm:zeemanyes}
Let $C, D$ be polytopal complexes such that $D \subset C$. Suppose that some subdivision $C'$ of $C$ collapses onto some subdivision $D'=\RS(C',D)$ of $D$. Then for $n$ large enough,
$\sd^{n\,} C \, \searrow \, \sd^{n\,} D.$
\end{cor}

\begin{proof}
With Lemma~\ref{lem:common}, let us choose $n$ large enough, so that $\sd^n C$ is a subdivision of $C'$. Then $\sd^{n+1} C \searrow \sd^{n+1} D$ by Theorem~\ref{thm:HC}.
\end{proof}

So, derived subdivisions induce collapsibility as well. We will now see that instead of subdividing, one could take repeated product with intervals: The effect is similar, and this is not a coincidence.

\setcounter{theorem}{3}
\subsection{Collapsibility of products}
In this section, we show that any contractible complex can be made collapsible by taking products with the $n$-dimensional cube, for $n$ suitably large.

In order to study products with cubes, we start by introducing a special subdivision, which is in some sense a cubical analogue of the classical operation of stellar subdivision.

\begin{definition}\label{def:tubu}
Let $C$ be a polytopal complex in $\R^d$, and let $\tau$ be any face of $C$. Let $v_\tau$ denote a point anywhere in the relative interior of $\tau$, and let $\lambda$ be any number in the interval $(0,1)$. Define
\[
\cst(\tau,C):=(C-\tau) \cup \{\conv \sigma\cup (\lambda\sigma+(1-\lambda)v_\tau) : \sigma \in \St(\tau,C)-\tau \}.
\]
$\cst(\tau,C)$ is the \Defn{cubical stellar subdivision}, or \Defn{c-stellar subdivision} of $C$ at $\tau$.
\end{definition}

\begin{lemma}\label{lem:tubeCollapse}
Let $C$ be an arbitrary complex, and let $\tau$ be any face of $C$. Then $C\times [0,1]$ collapses to (a complex combinatorially equivalent to) $\cst(\tau,C)$.  
\end{lemma}

\begin{figure}[htbf]
  \centering 
  \includegraphics[width=0.99\linewidth]{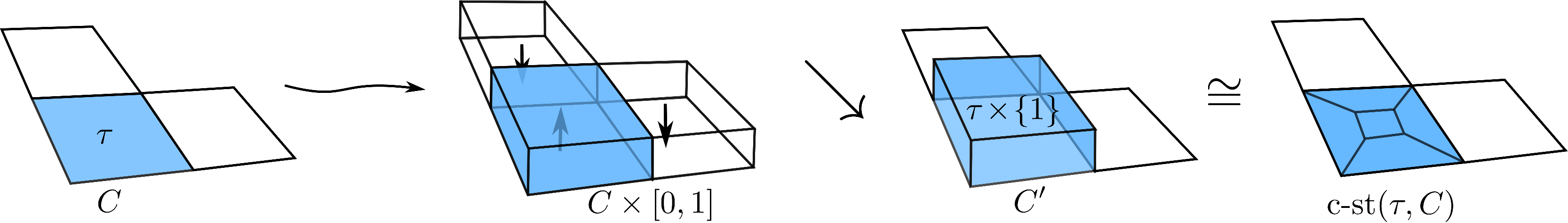} 
\caption{\small $C\times [0,1]$ collapses vertically to $C'$, which is combinatorially equivalent to a subdivision $\cst(\tau,C)$ of $C$.}
  \label{fig:dierkertrick}
\end{figure}

\begin{proof} Let $\sigma$ be a face of $C$. We perform collapses on $C \times [0,1]$ as follows:
\begin{compactitem}[$\circ$]
\item If $\sigma$ contains $\tau$, we delete $\sigma \times \{0\}$ from the complex.
\item If $\sigma$ is not in $\St(\tau,C)$, we delete $\sigma \times \{1\}$ from the complex
\end{compactitem}
The deletions are conducted in order of decreasing dimension, in order to ensure that faces are free when deleted. The result is combinatorially equivalent to $\cst(\tau,C)$, compare Figure~\ref{fig:dierkertrick}.
\end{proof}

\begin{definition}[Cubical derived subdivision]
A \Defn{cubical derived subdivision}, or \Defn{c-derived subdivision}, is any subdivision of a $d$-complex $C$ obtained by c-stellar subdivision, subdividing first all $d$-faces of $C$, then all $(d-1)$-faces of resulting complex, then all $(d-2)$-faces, and so on up to the faces of dimension~$1$.
\end{definition}

\noindent c-stellar and c-derived subdivisions mimic the classical notions of stellar and derived subdivisions, respectively. (Recall that a derived subdivision of a $d$-complex can be obtained by iterated stellar subdivision, beginning with the $d$-faces of $C$, then all $(d-1)$-faces of $C$, and so on.) In fact, c-derived subdivisions enjoy several analogous properties, among them the following theorem: 

\begin{theorem} \label{thm:tub} Let $C, D$ be polytopal complexes such that $D \subset C$. Suppose that some subdivision $C'$ of $C$ collapses onto some subdivision $D'=\RS(C',D)$ of $D$. Then for $n$ large enough,
$\csd^{n} C \, \searrow \, \csd^{n} D.$ \qed
\end{theorem}

The proof is analogous to the proof of Theorem \ref{thm:zeemanyes}. We are now ready to show that contractible complexes can be made collapsible by taking products with high-dimensional cubes. 
Recall that two complexes are called \Defn{simple homotopy equivalent} if there exists a sequence of PL expansions and PL collapses that transforms one complex into the other, cf.~\cite{CohenM}. 

\begin{theorem}\label{thm:simphomo}
If two polytopal complexes $C$, $D$ are simple homotopy equivalent, there exists an $n\ge 0$ such that $C\times \mathrm{I}^n$ collapses to a polytopal complex combinatorially equivalent to a subdivision of $D$.
\end{theorem}

\begin{proof}
Let $C$ and $D$ be two complexes that are simple homotopy equivalent. By a classical result of Dierker, Lickorish and Cohen \cite{Dierker, Lickorish, Cohen}, there is an $\ell\ge 0$ such that the complex $E:=C\times \mathrm{I}^{\ell}$ is PL equivalent to a complex that collapses to a complex $F$ PL homeomorphic to $D$.

By Theorem~\ref{thm:tub} there exists a c-derived subdivision $\csd^{m} E$ of $E$ that collapses to $\csd^{n} F=\RS(\csd^{m} E,F)$. Lemma~\ref{lem:tubeCollapse} now provides an $n\ge 0$ so that $ E\times \mathrm{I}^n$ collapses to a complex combinatorially equivalent to $\csd^{n} F$. Thus, $C\times \mathrm{I}^{m+\ell}= E \times \mathrm{I}^{m}$ collapses to a complex PL homeomorphic to $D$.
\end{proof}

\begin{cor}\label{cor:oliver}
For any contractible complex $C$, there is an $n\ge 0$ for which $C\times \mathrm{I}^n$ is collapsible.
\end{cor}

\begin{proof}
Every contractible complex is simple homotopy equivalent to a point, cf.\ \cite{Whitehead}.
\end{proof}

\begin{rem}\label{rem:shelloliver}
The analogous theorem does not hold with respect to shellability. Indeed, not even the results of Dierker, Lickorish and Cohen extend to shellability. For instance, let $C:=v\ast \mathrm{D}$ be the join of the Dunce's hat $\mathrm{D}$ with a vertex, and let $n$ be any nonnegative integer. Let $\varSigma=\varphi(C\times \mathrm{I}^n)$ be any complex PL homeomorphic to $C\times \mathrm{I}^n$, and let $\sigma$ be any facet of $\RS\left(\varSigma,\varphi(v\times \mathrm{I}^n)\right)$. Then $\Lk(\sigma,\varSigma)$ is PL homeomorphic to the Dunce's hat $\mathrm{D}$, and therefore not shellable. Hence $\varSigma =C\times \mathrm{I}^n$ is \emph{not} PL homeomorphic to a shellable complex.
\end{rem}

\paragraph*{Non-evasiveness} A stronger notion than collapsibility, called non-evasiveness, was introduced in \cite{KahnSaksSturtevant} in connection with the Aanderaa--Karp--Rosenberg conjecture. A simplicial $d$-complex $C$ is called \Defn{non-evasive} if either $d=0$ and $C$ is a single point, or $d >0$ and there is a vertex $v$ of $C$ so that both $\Lk (v,C)$ and $C-v$ are non-evasive. Every non-evasive complex is collapsible, and therefore contractible.
 To make sense of non-evasiveness with respect to products, recall that an \Defn{order complex} of a poset $\mathcal{R}$ is any simplicial complex $\varOmega(\mathcal{R})$ whose face poset is combinatorially equivalent to the poset of chains in $\mathcal{R}$, ordered by inclusion. A poset is \Defn{non-evasive} if its order complex is non-evasive, and the product of posets is defined as follows: If $\mathcal{A}$ and $\mathcal{B}$ are two posets on sets $A, B$ and respectively, then $\mathcal{A}\times \mathcal{B}$ is a poset on the set $A\times B$, where $(a,b)\prec(a',b')$ if and only if $a\prec a'$ and  $b\prec b'$.

\begin{prp}[{Welker~\cite{Welker}}]\label{prp:welker}
The product of any two non-evasive posets is non-evasive.
\end{prp}

Welker~\cite[Open Problem 2]{Welker} asked whether the converse of Proposition \ref{prp:welker} holds; more precisely, he asked whether the non-evasiveness of the product $\mathcal{A} \times \mathcal{B}$ of two order complexes $\mathcal{A}, \mathcal{B}$ implies that both $\mathcal{A}$ and $\mathcal{B}$ are non-evasive.

To see the connection to products with intervals, let us denote by $\mathcal{I}$ the poset on two elements $0,\ 1$, ordered by $0<1$. Clearly, $\varOmega(\mathcal{I})$ is combinatorially equivalent to $\mathrm{I}$. Moreover, products of posets with $\mathcal{I}$ behave like products of complexes with $\mathrm{I}$: Indeed $\varOmega(\mathcal{P}\times \mathcal{I})$ is a subdivision of $\varOmega(\mathcal{P})\times \varOmega(\mathcal{I})$ that introduces no new vertices. This connection allows us to answer Welker's problem.

\begin{prp} \label{prop:orderComplex}
There exists an evasive poset $\mathcal{P}$ such that $\mathcal{P} \times \mathcal{I}$ is non-evasive. 
\end{prp}

\begin{proof}

Let $\mathcal{R}$ denote the face poset of the polytopal complex given Figure \ref{fig:welker}(1). Let $\mathcal{P}$ denote the poset given by adding an element $v$ to $\mathcal{R}$, that is defined to be larger than all faces of the highlighted subcomplex in Figure \ref{fig:welker}(1). Then the order complex $C:=\varOmega(\mathcal{P})$ can be pictured as in Figure \ref{fig:welker}(2).

\begin{figure}[htbf]
  \centering 
  \includegraphics[width=0.7\linewidth]{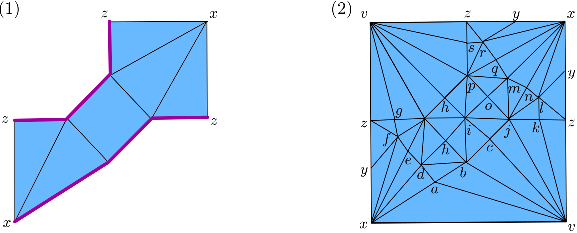} 
\caption{(1) A construction step for the poset $\mathcal{P}$. (2) The order complex $C=\varOmega(\mathcal{P})$}
  \label{fig:welker}
\end{figure}

The complex $C$ is collapsible (starting from the bottom edge), but the links of all vertices deformation retract to $1$-spheres, hence they are not contractible and cannot be non-evasive. Consequently, $C$ and in particular $\mathcal{P}$ are not non-evasive; compare~\cite[Example 5.4]{MiniaBarmak}.

Consider now the complex $\varOmega(\mathcal{P}\times \mathcal{I})$, in which every vertex $v$ of $C$ corresponds to two copies $(v,0)$ and $(v,1)$ of $\varOmega(\mathcal{P}\times \mathcal{I})$. A \Defn{non-evasive deletion} is the deletion of a vertex whose link is non-evasive. Clearly, a complex is non-evasive if it can be deformed to a point by non-evasive deletions, which we will show for $\varOmega(\mathcal{P}\times \mathcal{I})$.
\begin{compactenum}[(1)]
\item Delete all the vertices of $\varOmega(\mathcal{P}\times \mathcal{I})$ of type $(\cdot,v)$ that are not adjacent to $(d,1)$ in an order given by the poset. 
\item From the resulting complex, delete the vertices $(v,0)$ and $(x,1)$, in that order.
\item From the resulting complex, delete the vertices $(\cdot,0)$, starting with $(a,0)$ and ending up with $(s,0)$ in alphabetical order.
\item From the resulting complex, delete the vertices $(v,1)$, $(z,1)$, $(z,0)$ and $(y,0)$ in that order.\end{compactenum}
Only the vertex $(x,0)$ remains, and one can easily check the deletions described above are {non-evasive deletions}.
\end{proof}

\subsection{A simplicial approach to the Zeeman conjecture}
The \Defn{Zeeman conjecture} claims that the product of any contractible $2$-dimensional complex $C$ with the interval $[0,1]$ is PL homeomorphic to a collapsible complex. Its importance is highlighted by the fact that it implies, quite straightforwardly, the three-dimensional Poincar\'{e} conjecture, recently proven by Perelman~(cf.\ \cite[pp.\ 78--79]{Glaser}). In turn, Perelman's work straightforwardly shows the validity of the Zeeman conjecture for all ``special'' contractible  $2$-complexes that embed in $\R^3$, cf.\ \cite{Gillman, MatveevRolfsen}.

Using Theorem~\ref{mainthm:equivalence}, this (still open) conjecture can be rephrased in terms of derived subdivisions.

\begin{prp} \label{prop:CohenBary} 
The Zeeman conjecture is equivalent to the following conjecture: For any contractible $2$-complex $C$, there is an $m\ge 0$ such that $\sd^{m\,} (C \times \mathrm{I})$ is collapsible.
One can prove the following weaker statement: For any contractible $2$-complex $C$, there is an $m\ge 0$ such that $\sd^{m\,} (C \times \mathrm{I}^6)$ is collapsible. 
\end{prp}

\begin{proof}
The first claim follows from Theorem~\ref{thm:zeemanyes}. As for the second claim, Cohen showed for any contractible $2$-complex~$C$, the product $C \times \mathrm{I}^6$ is PL equivalent to a collapsible complex~\cite{Cohen}. The conclusion follows then from Theorem~\ref{thm:zeemanyes}.
\end{proof}

\noindent Our next goal is to show that the number of subdivisions needed, in order to achieve collapsibility, can be arbitrarily high.

\begin{lemma} \label{lem:3ball}
Let $B$ be any $3$-ball. Let $m, n$ be two non-negative integers. The link of every proper face in $\sd^{m} (B\times \mathrm{I}^n)$ is shellable.
\end{lemma}

\begin{proof} Since $B$ is three-dimensional, each link in $B$ is a sphere or a ball of dimension $\le 2$, and therefore shellable by Proposition \ref{prp:shs}. For brevity, set $A := B\times \mathrm{I}^n$ and $A':= \sd^m A$. 

For any face $\tau$~of $A$, there is a face $\sigma$ in $B$ such then $\Lk(\tau, A)$ is obtained from $\Lk(\sigma, B)$ by joins with $(n-1)$-simplices. Since joins preserve shellability by Lemma \ref{lem:jshell}, we obtain that $\Lk(\tau, A)$ is shellable. So also every link in $A$ is shellable.

Now, for any face $\tau'$ of $A'$, there is a face $\tau$ in $A$ such that  $\Lk(\tau', A')$ can be obtained from $\Lk(\tau, A)$ via iterated derived subdivisions and joins with polytope boundaries, where the polytopes are faces of~$A$. Polytope boundaries are shellable by the Bruggesser--Mani Theorem \cite{BruggesserMani}, and joins and derived subdivisions preserve shellability by Lemmas~\ref{lem:bary} and \ref{lem:jshell}. Hence every link in $A'$ is shellable. \end{proof}

\begin{definition}[cf.~\cite{Benedetti-DMT4MWB}]
A triangulation $C$ of a $d$-manifold with non-empty (resp.\ empty) boundary is called \Defn{endo-collapsible} if $C$ minus a $d$-face collapses onto $\parti  C$ (resp.\ onto a vertex).
\end{definition}  

\begin{lemma}[{\cite[Corollary~3.21]{Benedetti-DMT4MWB}}]
\label{lem:endcollapse}
Let $B$ be a collapsible PL $d$-ball. 
If $\sd\sm \Lk\,  (\sigma,B)$ is endo-collapsible for every proper face $\sigma$, then $\sd\sm B$ is also endo-collapsible. 
\end{lemma}

\begin{lemma} [{\cite[Theorem~3.1]{Benedetti-DMT4MWB}}]
\label{lem:shellableEndo}
All shellable manifolds are endo-collapsible. 
\end{lemma}

\begin{lemma}[{\cite[Theorem~3.15]{Benedetti-DMT4MWB}}]
\label{lem:homotbound}
Let $M$ be an endo-collapsible PL $d$-manifold. Let $L$ be a subcomplex of $M$, with $\dim L = \ell \le d-2$, such that all facets of $L$ lie in the interior of $M$. Then the homotopy group $\pi_{d-\ell-1} (|M| - |L|)$ has a presentation with (at most) $f_{\ell}(L)$ generators.
\end{lemma}

\begin{theorem}\label{thm:zeemanno}
For any $m, n \ge 0$, there is a $3$-ball $B$ such that $\sd^m (B \times \mathrm{I}^n)$ is not collapsible. 
\end{theorem}

\begin{proof}
We prove the claim in two parts:
\begin{compactenum}[(1)]
\item we prove that if $\sd^{m} (B \times \mathrm{I}^n)$ is collapsible, then $\sd^{m+1} (B\times \mathrm{I}^n)$ is endo-collapsible;
\item we construct a 3-ball $B$ such that $\sd^{m+1} (B\times \mathrm{I}^n)$ is not endo-collapsible.
\end{compactenum}
The conclusion follows immediately by combining (1) and (2). So, let us prove these claims.

\begin{compactenum}[(1)]
\item Via Lemma~\ref{lem:endcollapse}, it suffices to show that the link of every proper face in $\sd^{m+1} (B\times \mathrm{I}^n)$ is endo-collapsible. In Lemma~\ref{lem:3ball} we showed that any such link is shellable; by Lemma~\ref{lem:shellableEndo}, shellability implies endo-collapsibility.
\item Let $N = N(m,n)$ be the number of facets of the complex $\sd^{m+1} \, \mathrm{I}^{n+1}$. Let $B$ be a simplicial 3-ball with a  $3$-edge subcomplex $K$, isotopic to the connected sum of $3N$ trefoil knots. (For how to construct it, see \cite{HZ}.) Then any presentation of the group $\pi_{1}(|B|-|K|)$ must have at least $3N + 1$ generators, cf.~Goodrick~\cite{GOO}. 
Now, set $M:=\sd^{m+1} (B\times \mathrm{I}^n)$ and $L:=\sd^{m+1}(K \times \mathrm{I}^{n})$. This $L$ is $(n+1)$-dimensional, and has exactly $3 N$ facets, while $M$ is $(n+3)$-dimensional. Clearly $|B|- |K|$ is a deformation retract of $|M|-|L|$, so the homotopy groups of these two spaces are the same. In particular, any presentation of $\pi_{1}(|M|-|L|)$ must have at least $3N + 1$ generators. Were $M$ endo-collapsible, by Lemma~\ref{lem:homotbound} we would obtain a presentation of $\pi_{1}(|M|-|L|)$ with (at most) $3N$ generators, a contradiction.\qedhere
\end{compactenum}
\end{proof}

\begin{cor}\label{cor:zeemanno}
For any pair of non-negative integers $m, n$, there is a contractible 2-dimensional simplicial complex $C$ such that $\sd^m (C \times \mathrm{I}^n)$ is not collapsible. 
\end{cor}

\begin{proof}
Let $B$ be any $3$-ball for which $\sd^m (B \times \mathrm{I}^n)$ is not collapsible; for example, the one constructed in Theorem~\ref{thm:zeemanno}. Every polytopal $d$-ball collapses to a $(d-1)$-dimensional subcomplex, so that $B$ collapses to some $2$-complex $S$. The complex $S$ is a deformation retract of $B$ and therefore contractible. Moreover, since $B$ collapses onto $S$, $B \times \mathrm{I}^n$ collapses onto $S \times \mathrm{I}^n$ and $\sd^m (B \times \mathrm{I}^n)$ collapses onto $\sd^m (C \times \mathrm{I}^n)$. Were the latter complex collapsible, $\sd^m (B \times \mathrm{I}^n)$ would be collapsible as well, a contradiction.
\end{proof}

\begin{rem}
While the above Theorems use knot theory to provide complexes that are far from collapsible, one can also use the algorithmic undecidability. Indeed, were there universal constants $m, n$ such that for every contractible simplicial complex $C$ the complex $\sd^m (C \times \mathrm{I}^n)$ is collapsible, then by checking collapsibility of $\sd^m(C \times \mathrm{I}^n)$, which can be done by checking all possible collapsing sequences, we would have an algorithm to decide the contractibility of simplicial complexes, and in particular of triangulated manifolds.

Recall now that Novikov \cite{Novikov} proved that it is not decidable whether a given manifold with the homology of the $5$-sphere is actually the $5$-sphere. Now, a triangulated manifold $M$ with the homology of $S^5$ is homeomorphic to $S^5$ if and only if for any $5$-ball $B$ in $M$, the manifold with boundary $M{\sm\setminus\sm} B$ is contractible. Hence, deciding whether a given $5$-manifold is contractible is not decidable by any algorithm. This contradicts the soundness of the proposed algorithm above, and it follows that no universal constants $m$, $n$ exist such that contractibility of $C$ implies the collapsibility of $\sd^m(C \times \mathrm{I}^n)$.

Notice that this approach does not give us a $2$-complex, but simplicial complexes of higher dimension. It is not known to the authors whether the contractibility of $2$-complexes is decidable.
\end{rem}

 \paragraph*{Acknowledgments}
 We thank Micha\l{} Adamaszek for a useful suggestions and Anders Bj\"orner for informing us of Oliver's conjecture. A crucial part~of the research leading to this paper has been conducted during the first author's visit to the Hebrew University, Jerusalem. Both authors wish to thank the Institut Mittag-Leffler for the great hospitality during a summer program in 2012, organized jointly with Alexander Engstr\"om.

{\small 
\bibliographystyle{myamsalpha}
\bibliography{shell}
}

\end{document}